\documentclass[oneside,final,11pt]{amsart}

\usepackage{dsfont}
\usepackage[all]{xy}

\newcommand{\R}{\mathds R}

\DeclareMathOperator{\supp}{supp}

\title[Additional set-theoretic assumptions]{Additional set-theoretic assumptions and twisted sums of Banach spaces}
\author{Claudia Correa}
\thanks{The author is sponsored by FAPESP (Process no.\ 2014/00848-2).}
\address{Departamento de Matem\'atica,\hfill\break\indent Universidade de S\~ao Paulo, Brazil}
\email{claudiac.mat@gmail.com}

\begin{document}

\theoremstyle{plain}\newtheorem{teo}{Theorem}[section]
\theoremstyle{plain}\newtheorem{prop}[teo]{Proposition}
\theoremstyle{plain}\newtheorem{lem}[teo]{Lemma}
\theoremstyle{plain}\newtheorem{cor}[teo]{Corollary}
\theoremstyle{definition}\newtheorem{defin}[teo]{Definition}
\theoremstyle{remark}\newtheorem{rem}[teo]{Remark}
\theoremstyle{plain} \newtheorem{assum}[teo]{Assumption}
\theoremstyle{definition}\newtheorem{example}[teo]{Example}
\theoremstyle{plain}\newtheorem*{conjecture}{Conjecture}
\theoremstyle{plain}\newtheorem{question}{Question}

\maketitle

\begin{abstract}
In this paper, we discuss the role played by additional set-theoretic assumptions in the investigation of the existence of nontrivial twisted sums of $c_0$ and spaces of continuous functions on nonmetrizable compact Hausdorff spaces.
\end{abstract}

\begin{section}{History and Background of the problem}
In these notes we analyse the role played by some additional set-theoretic assumptions in the study of a classical problem in Banach space theory. This problem is related to the existence of nontrivial twisted sums of Banach spaces. We recall some basic definitions and facts.

\begin{defin}
Let $X$ and $Y$ be Banach spaces. A {\em twisted sum} of $Y$ and $X$ is an exact sequence in the category of Banach spaces, i.e., it is an exact sequence of the form:
\[0\rightarrow Y \rightarrow Z\rightarrow X\rightarrow0,\]
where $Z$ is a Banach space and the maps are bounded linear operators. This twisted sum is said to be {\em trivial} if the sequence splits, i.e., if the image of the map $Y\rightarrow Z$ is complemented in $Z$.
\end{defin}
For a nice discussion on exact sequences of Banach spaces, see \cite[Chapter I]{ThreeSpace}.
Given Banach spaces $X$ and $Y$, note that the direct sum of $Y$ and $X$ provides a twisted sum of $Y$ and $X$. More precisely:
\[0\rightarrow Y \overset{i_1}\longrightarrow Y \bigoplus X \overset{\pi_2}\longrightarrow X \rightarrow0\]
is a twisted sum, where the direct sum is endowed with a product norm, the map $i_1$ is the first inclusion and $\pi_2$ is the second projection. It is clear that this twisted sum is trivial. At this point one might wonder about the existence of nontrivial twisted sums of Banach spaces. Unlike the category of vector spaces, in which every twisted sum is trivial, in the category of Banach spaces there are examples of nontrivial twisted sums.
A classical example is provided by an old result of Phillips \cite{Phillips} which states that the space $c_0$ of real sequences converging to zero is not complemented in $\ell_\infty$, the space of bounded real sequences. Therefore, the following twisted sum is nontrivial:
\[0 \rightarrow c_0 \rightarrow \ell_\infty \rightarrow \ell_\infty/c_0 \rightarrow 0,\]
where the arrows are the inclusion and the quotient map. However there are pairs of Banach spaces that admit only trivial twisted sums. Interesting examples of this phenomenon are consequences of the classical Theorem of Sobcyk \cite{Sobczyk}, that states that $c_0$ is complemented in every separable superspace; more precisely, every isomorphic copy of $c_0$ inside a separable Banach space is complemented. Since separability is a three space property \cite{ThreeSpace}, it follows that if $X$ is a separable Banach space, then every twisted sum of $c_0$ and $X$ is trivial. In these notes we are interested on the converse of this last implication. In other words:
If $X$ is a Banach space such that every twisted sum of $c_0$ and $X$ is trivial, then $X$ must be separable? This question is easily answered negatively since there are nonseparable projective Banach spaces; namely, the spaces $\ell_1(\Gamma)$, for any uncountable set $\Gamma$ \cite[Lemma 1.4.a]{ThreeSpace}. However this question becomes interesting when we restrict ourselves to the class of $C(K)$ spaces.
Here $C(K)$ denotes the Banach space of continuous real-valued functions defined on a compact Hausdorff space $K$, endowed with the supremum norm. Recall that $C(K)$ is separable if and only if $K$ is metrizable. Therefore Sobczyk's Theorem implies that if $K$ is a compact metrizable space, then every twisted sum of $c_0$ and $C(K)$ is trivial. In this context, the converse we are discussing can be phrased as the following question.
\begin{question}\label{bigQuestion}
Is there a compact Hausdorff nonmetrizable space $K$ such that every twisted sum of $c_0$ and $C(K)$ is trivial?
\end{question}
This question was first stated in the papers \cite{CastilloKalton, JesusSobczyk}. Until last year, there were few results related to this problem. They were summarized in \cite[Proposition~2]{Castilloscattered}, namely we have the following proposition.

\begin{prop}\label{resumoResultados}
Let $K$ be a compact Hausdorff space. There exists a nontrivial twisted sum of $c_0$ and $C(K)$ under any one of the following assumptions:
\begin{enumerate}
\item $K$ is a nonmetrizable Eberlein compact space;
\item $K$ is a Valdivia compact space which does not satisfy the countable chain condition (ccc);
\item the weight of $K$, denoted by $w(K)$, is equal to $\omega_1$ and the dual space of $C(K)$ is not weak*-separable;
\item $K$ has the extension property (\cite{CTextensor}) and it does not have ccc;
\item $C(K)$ contains an isomorphic copy of $\ell_\infty$; in particular, it is the case if $K$ is infinite and extremally disconnected.
\end{enumerate}
\end{prop}

In a recent series of papers \cite{Castilloscattered, ExtCKc0}, this problem was extensively studied and great progress was achieved towards its solution. In these works the importance of additional set-theoretic assumptions in the understanding of Question \ref{bigQuestion} became clear. Finally, in \cite{Plebanek} the first consistent examples of compact Hausdorff nonmetrizable spaces $K$ such that $c_0$ can not be nontrivially twisted with $C(K)$ were given. As we will discuss in Section~\ref{sec:Resultados}, thanks to the results of \cite{Castilloscattered} and \cite{ExtCKc0} the existence of nontrivial twisted sums of $c_0$ and $C(K)$, for some of the spaces $K$ given by Marciszewski and Plebanek in \cite{Plebanek}, are independent of the axioms of ZFC. The additional set-theoretic assumptions used in those works are the Continuum hypothesis and Martin's axiom.
We observe that the existence, in ZFC, of a compact Hausdorff nonmetrizable space $K$ such that $c_0$ can not be nontrivially twisted with $C(K)$ is still an open problem. In Section~\ref{sec:Resultados}, we discuss the results obtained in \cite{Castilloscattered}, \cite{ExtCKc0} and \cite{Plebanek}, and we give more details about those obtained by myself and D. V. Tausk in \cite{ExtCKc0}. In the final section, we describe the ongoing investigation of some open problems related to Question \ref{bigQuestion} and the progress achieved towards their solutions in \cite{ValdiviaTree}, where we worked assuming the Diamond axiom. In Section \ref{sec:axioms}, we present a brief introduction to the additional set-theoretic assumptions used in the works mentioned above.

\end{section}

\begin{section}{Continuum hypothesis, Diamond and Martin's axioms}
\label{sec:axioms}

The Continuum hypothesis (CH) was born in the early history of modern set theory; actually CH stimulated the development of this theory. It all began with G. Cantor's investigation of different kinds of infinities \cite{Cantor}. Cantor showed that there is no one-to-one correspondence between the set of natural numbers and the set of real numbers. In other words, using the idea that the existence of a one-to-one correspondence between two sets expresses the notion of having the same number of elements, Cantor showed that there are strictly more real than natural numbers. In this context, he started investigating the existence of a set with (strictly) more elements than the natural numbers and (strictly) less elements than the real numbers. This investigation led to the modern formulation of CH. We denote by $\omega$ the first infinite ordinal, by $\omega_1$ the first uncountable ordinal and by $\mathfrak{c}$ the cardinality of the real numbers, also known as the cardinality of the continuum. The Continuum hypothesis is the following statement: there is no cardinal number strictly between $\omega$ and $\mathfrak c$, i.e., $\mathfrak c=\omega_1$. Cantor was never able to prove or disprove CH. It was only with the works of K. G\"odel \cite{Godel} and P. Cohen \cite{Cohen} that we were able to understand the status of CH regarding the axiomatic set-theory we use to do modern mathematics. Those mathematicians showed that CH is independent of ZFC, i.e., both CH and its negation are consistent with ZFC. Therefore, Cantor would never be able to prove or disprove CH, using the axioms of ZFC.

Now we turn our attentions to a classical problem in set theory that motivated the next axioms we want to discuss here: Diamond axiom and Martin's axiom. In his paper \cite{Cantor}, Cantor characterized the canonical order of the real numbers; namely, he proved that every nonempty totally ordered set with no endpoints that is connected and separable, when endowed with the order topology, is order-isomorphic to the real line. The {\em Souslin problem} asks if we can replace separable with ccc in the above statement. More precisely, the Souslin problem is the following question: If $X$ is a nonempty totally ordered set without endpoints which is connected and satisfies ccc, when endowed with the order topology, then $X$ is order-isomorphic to the real line? The {\em Souslin hypothesis} (SH) is the statement that the Souslin problem has positive answer. The first appearance of the Souslin problem was in an early volume of Fundamenta Mathematicae in 1920 as part of a list of open problems and it was attributed to  Mikhail Souslin. The first great breakthrough on the investigation of SH happened when S. Tennenbaum \cite{Tennenbaum} showed the relative consistency of the negation of SH with ZFC. Later, R. Jensen \cite{Jensen} gave another proof of the consistency of the negation of SH. Jensen isolated an interesting combinatorial principle (consistent with ZFC) that implies the negation of SH, this principle is called the {\em Diamond Axiom} ($\diamondsuit$). Finally, in 1971, R. Solovay and Tennenbaum established the consistency of SH \cite{Solovay}. Martin's axiom (MA) was then isolated by T. Martin from this work of Solovay and Tennenbaum.

In what follows we discuss briefly the statements of $\diamondsuit$ and of MA.
We start with $\diamondsuit$. In order to understand this combinatorial principle, we need to recall some basic concepts. This axiom is related to, in some sense, big subsets of $\omega_1$. A subset of $\omega_1$ is said to be a {\em club} if it is closed and unbounded in $\omega_1$, where $\omega_1$ is endowed with the canonical order topology; we say that it is {\em stationary} if it intersects every club.

\begin{defin}
Let $I$ be a set.
\begin{enumerate}
\item An {\em $\alpha$-filtration} of $I$, where $\alpha$ is an ordinal number, is an increasing family $(I_\beta)_{\beta \in \alpha}$ of subsets of $I$ such that $I=\bigcup_{\beta \in \alpha}I_\alpha$. This filtration of $I$ is said to be {\em continuous} if, for every limit ordinal $\beta \in \alpha$, we have $I_\beta=\bigcup_{\lambda \in \beta}I_\lambda$.
\item Given an $\omega_1$-filtration of $I$, a {\em diamond family} for this filtration is a family $(I_{\alpha}^\diamondsuit)_{\alpha \in \omega_1}$ with each $I_\alpha^\diamondsuit$ a subset of $I_\alpha$ and such that given any subset $J$ of $I$, the set
\[\{\alpha \in \omega_1: J \cap I_\alpha=I_\alpha^\diamondsuit\}\]
is stationary.
\end{enumerate}
\end{defin}
The axiom $\diamondsuit$ is the statement that there exists a diamond family for every continuous $\omega_1$-filtration $(I_\alpha)_{\alpha \in \omega_1}$ of a set $I$ such that each $I_\alpha$ is countable. It is easy to see that $\diamondsuit$ is equivalent to the statement that there exists a diamond family for the canonical $\omega_1$-filtration of the set $\omega_1$, i.e., with $I_\alpha=\alpha$, for every $\alpha \in \omega_1$.
Moreover, since there exist $\omega_1$ disjoint stationary subsets of $\omega_1$ \cite[Chapter II, Corollary 6.12]{Kunen},
we have that the version of $\diamondsuit$ which requires the existence of a diamond family only when the set I is countable is equivalent to CH.

We finish this section by discussing a bit about MA. In order to state Martin's axiom, we need to introduce some concepts regarding partial orders. Let $(\mathbb P, \le)$ be a partial order. A subset $D$ of $\mathbb P$ is called {\em dense} in $\mathbb P$ if given $p \in \mathbb P$ there exists $d \in D$ with $d \le p$. Two elements $p$ and $q$ of $\mathbb P$ are said {\em compatible} if there exists $r \in \mathbb P$ with $r \le p$ and $r \le q$. If $p$ and $q$ are not compatible, we say that they are {\em incompatible}. An {\em antichain} of $\mathbb P$ is a subset of $\mathbb P$ with the property that any two distinct elements are incompatible; a partial order $(\mathbb P,\le)$ is said to satisfy the {\em countable chain condition} (ccc) if every antichain of $\mathbb P$ is countable. Since MA assures the existence of some filters on partial orders satisfying ccc, lets define those important objects.

\begin{defin}
A nonempty subset $G$ of $\mathbb P$ is said to be a {\em filter} if it satisfies the following properties:
\begin{itemize}
\item if $p$ and $q$ belong to $G$, then there exists $r \in G$ with $r \le p$ and $r \le q$, i.e., each pair of elements of $G$ are compatible and this fact is testified by some element of $G$;
\item if $p \in G$ and $q \in \mathbb P$ satisfies $p \le q$, then $q \in G$.
\end{itemize}
\end{defin}

Now we can state MA. For each infinite cardinal $\kappa$, MA($\kappa$) is the following statement: Let $(\mathbb P, \le)$ be a nonempty partial order and $\mathfrak D$ be a family of dense subsets of P. If $\mathbb P$ satisfies ccc and the cardinality of $\mathfrak D$ is at most $\kappa$, then there exists a filter in $\mathbb P$ that intersects every element of $\mathfrak D$. Finally, MA is the statement that MA($\kappa$) holds for $\omega \le \kappa< \mathfrak c$.
Since MA($\omega$) is a theorem of ZFC and MA($\mathfrak c$) is false \cite[Lemma 2.6]{Kunen}, we have that MA is interesting only if we assume the negation of CH.

\end{section}

\begin{section}{Consistency and independency results}
\label{sec:Resultados}

The class of nonmetrizable compact Hausdorff spaces contains dramatically distinct subclasses of spaces: we have classes that are ``close'' to the class of metrizable spaces, in the sense that their elements have many properties implied by metrizability and we have classes that are ``far'' from the class of metrizable spaces, i.e., their elements have little in common with the metrizable ones. For instance, with respect to sequential properties, the class of Corson compacta is close to the metrizable ones. In fact, every Corson compact space is a Frech\'et--Urysohn space \cite[Lemma~1.6 (ii)]{Kalenda}. On the other hand, the extremally disconnected compact spaces are completely different from the metrizable ones: they have no nontrivial convergent sequences \cite[Theorem 18]{Arhangelskii}.
In light of Sobczyk's Theorem, it is reasonable to believe that if there exists a nonmetrizable compact Hausdorff space $K$ such that $c_0$ can not be nontrivially twisted with $C(K)$, then it will belong to classes close to the class of metrizable compact spaces. Having those considerations in mind, D.V. Tausk and myself decided to investigate the twisted sums of $c_0$ and $C(K)$, for $K$ belonging to the class of Corson compact spaces and, more generally, its superclass of Valdivia compacta. To discuss our results, we start by recalling some standard definitions and well-known facts. Given an index set $I$, we write $\Sigma(I)=\big\{x\in\R^I:\text{$\supp x$ is countable}\big\}$, where the support $\supp x$ of $x$ is defined by $\supp x=\big\{i\in I:x_i\ne0\big\}$.

\begin{defin}
Given a compact Hausdorff space $K$, we call $A$ a {\em $\Sigma$-subset\/} of $K$ if there exists an index set $I$ and a continuous injection $\varphi:K\to\R^I$ such that $A=\varphi^{-1}[\Sigma(I)]$. The space $K$ is called a {\em Valdivia compactum\/} if it admits a dense $\Sigma$-subset and it is called a {\em Corson compactum\/} if $K$ is a $\Sigma$-subset of itself.
\end{defin}
It is clear that every compact metric space is Corson and that every Corson space is Valdivia (for an amazing survey on Corson and Valdivia compacta, see \cite{Kalenda}). In \cite{ExtCKc0}, we developed a technique for constructing nontrivial twisted sums of $c_0$ and certain nonseparable Banach spaces, using the existence of interesting biorthogonal systems. Using these techniques, we were able to solve the problem for Corson compacta assuming CH.

\begin{teo}\label{corson}
If $K$ is a Corson compact space with weight greater or equal to $\mathfrak{c}$, then there exists a nontrivial twisted sum of $c_0$ and $C(K)$.
In particular, under CH, there exists a nontrivial twisted sum of $c_0$ and $C(K)$, for every nonmetrizable Corson compact space $K$.
\end{teo}
\begin{proof}
See \cite[Theorem 3.1]{ExtCKc0}.
\end{proof}

It is known that, under MA and the negation of CH, every ccc Corson compactum is metrizable  \cite{ArgyrosMartin}. Having in mind Proposition \ref{resumoResultados} item $(2)$, we have that, under MA and the negation of CH, there exists a nontrivial twisted sum of $c_0$ and $C(K)$, for every nonmetrizable Corson compact space $K$. Therefore, it follows from Theorem \ref{corson} that, under MA, there exists a nontrivial twisted sum of $c_0$ and $C(K)$, for every nonmetrizable Corson compactum $K$. In this context, the following question remains open.

\begin{question}
Does it hold, in ZFC, that there exists a nontrivial twisted sum of $c_0$ and $C(K)$, for every nonmetrizable Corson compact $K$?
\end{question}

The general Valdivia case, under CH, remains open, but many results were obtained in \cite{ExtCKc0}. They are summarized in the next theorems. Recall that given a point $x$ of a topological space $\mathcal X$, we define the {\em weight of $x$ in $\mathcal X$\/} by:
\[w(x,\mathcal X)=\min\big\{w(V):\text{$V$ neighborhood of $x$ in $\mathcal X$}\big\}.\]

\begin{teo}\label{pontoGdelta}
Let $K$ be a Valdivia compact space admitting a $G_\delta$ point $x$ with $w(x,K) \ge \mathfrak{c}$. Then there exists a nontrivial twisted sum of $c_0$ and $C(K)$. In particular, under CH, if $K$ is a Valdivia compact space admitting a $G_\delta$ point with no second countable neighborhoods, then there exists a nontrivial twisted sum of $c_0$ and $C(K)$.
\end{teo}

\begin{teo}\label{convergentefoarSigma}
Assume CH. Let $K$ be a Valdivia compact space admitting a dense $\Sigma$-subset $A$ such that some point of $K\setminus A$ is the limit of a nontrivial sequence in $K$. Then there exists a nontrivial twisted sum of $c_0$ and $C(K)$.
\end{teo}

Regarding twisted sums of $c_0$ and spaces of continuous functions, a particular family of Valdivia compact spaces was recently shown to be special; namely the spaces $2^{\kappa}$. In \cite{ExtCKc0}, Tausk and myself presented the following result.

\begin{teo}\label{ValdComNOt}
If $\kappa$ is a cardinal number with $\kappa \ge \mathfrak c$, then there exists a nontrivial twisted sum of $c_0$ and $C(2^\kappa)$.
\end{teo}
\begin{proof}
See \cite[Corollary~2.10]{ExtCKc0}.
\end{proof}

Surprisingly, Marciszewski and Plebanek showed that, given a cardinal number $\kappa<\mathfrak{c}$, if MA($\kappa$) holds, then every twisted sum of $c_0$ and $C(2^\kappa)$ is trivial \cite[Corollary 5.2]{Plebanek}. In particular, under MA and the negation of CH, if $\kappa$ is a cardinal number satisfying $\kappa<\mathfrak{c}$, then every twisted sum of $c_0$ and $C(2^\kappa)$ is trivial. This answers consistently Question~\ref{bigQuestion}. Note that, under CH, Theorem~\ref{ValdComNOt} states that there exists a nontrivial twisted sum of $c_0$ and $C(2^\kappa)$, for every uncountable cardinal $\kappa$. Therefore the existence of nontrivial twisted sums of $c_0$ and $C(2^{\omega_1})$ is independent of the axioms of ZFC. The problem of determining if Question~\ref{bigQuestion} can be answered in ZFC remains open.

\begin{question}\label{aindaaberta}
Is there, in ZFC, a compact Hausdorff nonmetrizable space $K$ such that every twisted sum of $c_0$ and $C(K)$ is trivial?
\end{question}

To finish this section, we would like to tell the reader about the results of \cite{Castilloscattered}. In this work, Castillo showed that, assuming CH, if $K$ is a nonmetrizable compact Hausdorff space with finite Cantor--Bendixson height, then there exists a nontrivial twisted sum of $c_0$ and $C(K)$ \cite[Theorem 1]{Castilloscattered}.
It is worth commenting that, unlike the results in \cite{ExtCKc0}, where the nontrivial twisted sums were constructed, Castillo did not construct his nontrivial twisted sums; their existence is established by counting arguments (see \cite[Lemma 2]{Castilloscattered}). Interestingly, Plebanek and Marciszewski showed that, under MA($\kappa$), if $K$ is a separable scattered space of height 3 and weight $\kappa$, then every twisted sum of $c_0$ and $C(K)$ is trivial \cite[Theorem 9.7]{Plebanek}.

\end{section}

\begin{section}{Towards the answer to Question~\ref{aindaaberta}}
\label{sec:openProblems}

It follows from the discussion in Section~\ref{sec:Resultados} that Question~\ref{aindaaberta} can be rephrased as follows.

\begin{question}
Is there an additional consistent set-theoretic assumption that assures the existence of nontrivial twisted sums of $c_0$ and $C(K)$, for every nonmetrizable compact Hausdorff space $K$?
\end{question}

Since the first examples of nonmetrizable compact Hausdorff spaces such that $c_0$ can not be nontrivially twisted with their spaces of continuous functions were given assuming MA and the negation of CH, one might wonder if, under CH, there exists a nontrivial twisted sum of $c_0$ and $C(K)$, for every nonmetrizable compact Hausdorff space $K$. Having this consideration in mind and continuing the work of \cite{ExtCKc0}, D. Tausk and myself are currently investigating the following question.

\begin{question}
Is it true that, under CH, there exists a nontrivial twisted sum of $c_0$ and $C(K)$, for every nonmetrizable Valdivia compactum?
\end{question}

The techniques developed in \cite{ExtCKc0} provided nontrivial twisted sums of $c_0$ and $C(K)$, for a huge class of Valdivia compact spaces $K$. Note that Theorems \ref{pontoGdelta} and \ref{convergentefoarSigma} do not solve the problem, under CH, if $K$ is a nonempty Valdivia compactum satisfying all the following properties:
\begin{enumerate}
\item $K$ satisfies ccc;
\item $K$ does not admit a $G_\delta$ point;
\item $K$ does not admit a nontrivial convergent sequence in the complement of a dense $\Sigma$-subset.
\end{enumerate}
Note that the case when $K$ does not satisfy ccc is handled by Proposition~\ref{resumoResultados}(2). Finding examples of nonempty Valdivia compact spaces with no $G_\delta$ points and no nontrivial convergent sequences in the complement of a dense $\Sigma$-subset is not a trivial task, since the absence of $G_\delta$ points tends to make the complement of dense $\Sigma$-subsets ``large'' (see, for instance, \cite[Theorem~3.3]{Kalenda} for a more precise statement). In \cite[Proposition~4.7]{ExtCKc0}, it was shown that the path space of a certain tree $T$, endowed with the product topology of $2^T$, provides such an example. However, using this topology it is not possible to have a nonempty path space with no $G_\delta$ points and ccc. In \cite{ValdiviaTree}, D. Tausk and myself constructed an example of a nonempty Valdivia compact space satisfying Properties (1), (2) and (3) described above. This construction is done under $\diamondsuit$. This space, given in \cite[Theorem~4.1]{ValdiviaTree}, is the path space of a tree, endowed with an intricate compact Hausdorff topology. In what follows, we describe briefly the tools used in \cite{ValdiviaTree}.

Recall that a {\em tree\/} is a partially ordered set $(T,{\le})$ such that, for all $t\in T$, the set $(\cdot,t)=\big\{s\in T:s<t\big\}$ is well-ordered. A subset $X$ of $T$ is called an {\em initial part\/} of $T$ if $(\cdot,t)\subset X$, for all $t\in X$; a {\em chain\/} if it is totally ordered; an {\em antichain\/} if any two distinct elements of $X$ are incomparable; a {\em path\/} if it is both a chain and an initial part of $T$; the {\em path space} of a tree is the set of its paths. We say that $T$ satisfies the {\em countable chain condition\/} (ccc) if every antichain in $T$ is countable.
At this point, the reader must be wondering: what is the relationship between trees and Valdivia compact spaces?
It is a good question that is answered by the following facts:
\begin{itemize}
\item Kubi\'s and Michalewski established in \cite{KubisSmall} a correspondence between Valdivia compact spaces with weight at most $\omega_1$ and certain inverse limits of compact metric spaces \cite[Theorem~2.9]{ValdiviaTree};
\item we established a correspondence between those inverse limits and certain inverse limits of path spaces of trees \cite[Proposition~3.3]{ValdiviaTree}.
\end{itemize}

Therefore, combining those two correspondences, we obtain a characterization of Valdivia compact spaces with weight at most $\omega_1$ in terms of trees with some additional structures and suitable topologies on their path spaces \cite[Theorem~3.4]{ValdiviaTree}.
This characterization allows one to fine-tune the structure of a Valdivia compactum by manipulating the properties of the corresponding tree. We observe that axiom $\diamondsuit$ is used in the construction of our tree in a similar way that it is used to construct a Souslin tree. Recall that a {\em Souslin tree} is a tree with height $\omega_1$, satisfying ccc and admitting only countable paths (see \cite[Chapter II, §4]{Kunen} to understand the relationship between Souslin trees and the Souslin hypothesis).

\end{section}

\end{document}